\documentclass[12pt]{amsart}
\usepackage[english]{babel}
\usepackage[latin1]{inputenc}
\usepackage{amssymb,amsmath,amsthm}
\usepackage{mathrsfs}
\usepackage[mathscr]{eucal}
\usepackage[all]{xy}
\usepackage{multicol}

\newcommand{\BibTeX}{{\scshape Bib}\kern-.08em\TeX}
\newcommand{\T}{\S\kern .15em\relax}

\theoremstyle{plain}
\newtheorem{definition}{Definition}

\newtheorem{lemma}{Lemma}
\newtheorem{proposition}{Proposition}
\newtheorem{theorem}{Theorem}

\theoremstyle{definition}
\newtheorem{remark}{Remark}

\newtheorem{examples}{Examples}

\def\dim{\mathrm{dim}}

\begin{document}

\title{On a theorem of Faltings on formal functions}
\author{Paola Bonacini}
\address{Dipartimento di Matematica e Informatica\\
Università degli Studi di Catania\\
Viale A. Doria 6\\
95125 Catania}
\email{bonacini@dmi.unict.it}
\author{Alessio Del Padrone}
\address{Dipartimento di Matematica\\
Via Dodecaneso 35,
16146 Genova}
\email{delpadro@dima.unige.it}
\author{Michele Nesci}
\address{Dipartimento di Matematica\\
Universita' degli Studi Roma Tre\\
Largo San Leonardo Murialdo 1\\
00146 Roma}
\email{nesci@mat.uniroma3.it}
\date{}

\maketitle

\begin{abstract}
In 1980, Faltings proved, by deep local algebra methods, a local result regarding
formal functions which has the following global geometric fact  as a consequence.
{\it Theorem. $-$   Let $k$ be an algebraically closed field  (of any characteristic).
Let $Y$ be a closed subvariety of a projective irreducible variety $X$ defined over $k$.
Assume that $X\subset \mathbb P^n$, $\dim (X)=d>2$ and $Y$ is the intersection of
$X$ with $r$ hyperplanes of $\mathbb P^n$, with $r\le d-1$.
Then, every formal rational function on $X$ along $Y$
can be (uniquely) extended to a rational function on $X$.}
Due to its importance,  the aim of this paper is to provide two elementary global geometric
proofs of this theorem.
\end{abstract}

\section*{introduction}

The aim of this work is to give two elementary global geometric proofs of the following Theorem
\ref{FaltingsTheorem}, which is a consequence of a local result of Faltings
\cite{Fa80} by means of the general local-global philosophy explained in \cite{GroSGA2}.
Faltings original proof is not so easy to follow, and it is also not immediate that what he proved implies
Theorem \ref{FaltingsTheorem}, which is, on the other hand, useful for the applications.
Hence we think that giving elementary arguments could be of interest.

\begin{theorem}\label{FaltingsTheorem}
Let $k$ be an algebraically closed field  (of any characteristic).
Let $Y$ be a closed subvariety of a projective irreducible subvariety $X$ defined over $k$.
Assume that $X\subset \mathbb P^n$, $\dim (X)=d>2$ and $Y$ is the intersection of
$X$ with $r$ hyperplanes of $\mathbb P^n$, with $r\le d-1$. Then $Y$ is $G3$ in $X$.
\end{theorem}

If $Y$ is a complete intersection of $X$ with the hyperplanes $H_i$ (i.e. $\dim(Y)=d-r$),
this result was already proved geometrically by Hironaka and Matsumura \cite[(4.3), (3.5)]{HiMa68}
(see also \cite[9.25]{Bad}).
Faltings' result, as formulated in Theorem \ref{FaltingsTheorem}, is indeed useful in some applications.
For instance, B\u adescu used it in an essential way to prove a relevant strengthening of Fulton-Hansen connectedness Theorem (see \cite[(0.1)]{BaNag96}, cf. also \cite[Chapter 11]{Bad}, or also Example \ref{FirstExamples} below).\\

In this work we present two global geometric proofs of \ref{FaltingsTheorem}.
Both proofs use, repeatedly, as a key tool a Theorem of Hironaka and Matsumura
 (\cite[(2.7)]{HiMa68}, see also Theorem  \ref{HironakaMatsumuraFormula} below)
whose proof makes essential use of Grothendieck's existence Theorem \cite[(5.1.4)]{EGA}).
Indeed this result, under the hypotheses of Theorem \ref{FaltingsTheorem} and by means of geometric constructions yielding suitable morphisms, reduces the problem to the case of closed connected subvarieties of projective spaces of positive dimension.
This case is well known by another fundamental Theorem of Hironaka and Matsumura (\cite[(3.3)]{HiMa68}, see also Theorem  \ref{HironakaMatsumuraTheorem} below).
The first one is inspired by the proof of  Hironaka and Matsumura of the result, for the complete intersection case, mentioned above;
in fact what we do is to provide the necessary extra arguments (not completely trivial)
 in order to make Hironaka-Matsumura's proof work.
As far as the second proof is concerned, we use a standard construction (involving a certain incidence variety).\\

The first section is devoted to recall some basic facts of formal geometry and then to explain the common stategy of the two proofs, which are based on two ``projective geometry style'' constructions.
The first proof is given in section \ref{FirstProof}, and the second is presented in section
\ref{SecondProof}.




\medskip

{\bf Acknowledgments.}
We want to thank Professor Lucian B\u{a}descu
who proposed this problem to us and supervised our work, giving also several useful
comments and suggestions. This project started during our stay in the
PRAGMATIC 2006 held at the Department of Mathematics of the University
of Catania. We would also like to thank this institution for their warm
hospitality and for the excellent working atmosphere.

\section{Background material}\label{BackgroundSection}

The main reference is the original work of Hironaka and Matsumura \cite{HiMa68},
further material, together with a gentle introduction to ``formal geometry'', can be found
in the textbook \cite{Bad}.

Let $\mathcal{Z}$ be a formal scheme (see \cite{EGAI}, cf. also \cite{HaASAV70, Ha, Bad}).
The ring of formal rational functions on $\mathcal{Z}$, $K(\mathcal{Z})$, is defined as
follows: for any affine open subset $U$ of $\mathcal{Z}$, let $[\mathcal{O}_\mathcal{Z}(U)]_0$ be
the total ring of fractions of $\mathcal{O}_\mathcal{Z}(U)$, and let $\mathcal{M}_\mathcal{Z}$ be the sheaf associated
to the presheaf on $\mathcal{Z}$ defined by $U\mapsto [\mathcal{O}_\mathcal{Z}(U)]_0$.
Then $K(\mathcal{Z})=H^0(\mathcal{Z},\mathcal{M}_\mathcal{Z})$.
If $\mathcal{Z}$ is an ordinary scheme, $K(\mathcal{Z})$ is nothing but the usual ring of rational functions of
$\mathcal{Z}$.

As a special case, for any pair $(X,Y)$ with $X$ a locally noetherian scheme and $Y$ a connected closed subscheme of $X$,  we can consider the formal completion $\mathcal{Z}:=X_{/Y}$ of $X$ along $Y$.
Then the completion morphism $X_{/Y}\longrightarrow X$ gives rise to the canonical homomorphisms $H^0(X,\mathcal{O}_X)\longrightarrow H^0(X_{/Y},\mathcal{O}_{X_{/Y}})$ and
$K(X)\longrightarrow K(X_{/Y})$.

We recall that for any open neighbourhood $U$ of $Y$ in $X$ we have an isomorphism of formal schemes
$U_{/Y}\cong X_{/Y}$. Moreover: $ X_{/Y}\cong  X_{/(Y)_{\mathrm{red}}}$, that is $X_{/Y}$ depends just on the {\it closed subset} $Y$.

By \cite[Remark p. 57]{HiMa68}, when $X$ is a
reduced algebraic scheme (for example an algebraic variety),
then $K(X_{/Y})$ is a finite direct product of fields. If $Y$ is the disjoint union of two closed
subsets $Y_1$ and $Y_2$, then $K(X_{/Y})=K(X_{/Y_1})\times K(X_{/Y_2})$. Hence, in order
$K(X_{/Y})$ to be a field we must have $Y$ connected. Conversely,
if $X$ is an irreducible normal projective variety, and $Y$ is connected,
then $K(X_{/Y})$ is a field.

\begin{examples}\label{ExamplePn}
As basic examples we consider the case of (connected) subvarieties $Y$ of $X={\mathbb P}^n$,
the complex projective space.
\begin{itemize}
\item[a)] For $Y={P}\in X$ a point, say over $k={\mathbb C}$, we have
$X_{/Y}={\mathbb P}^n_{/P}\cong {\mathbb A}^n_{/(0,\dots, 0)}$,
and hence $K({\mathbb P}^n_{/P})\cong {\mathbb C}((x_1,\dots,x_n))$
(the field of fraction of the ring of formal power series ${\mathbb C}[[x_1,\dots,x_n]]$).\medskip
\item[b)] In case $Y$ is connected and positive dimensional, a fundamental result of Hironaka and Matsumura, quoted below (see Theorem  \ref{HironakaMatsumuraTheorem}), asserts that
$K({\mathbb P}^n_{/Y})\cong K({\mathbb P}^n)={\mathbb C}(x_1,\dots,x_n)$.
\end{itemize}
\end{examples}

A basic property of the ring of formal rational functions is given by the following useful
formula due to Hironaka and Matsumura (see \cite[(2.7)]{HiMa68}, cf. also \cite[9.11]{Bad}):

\begin{theorem}\label{HironakaMatsumuraFormula}
Let $f\colon X^\prime\longrightarrow X$ be a proper surjective morphism of
irreducible algebraic varieties, and let $Y\subseteq X$ be a closed subvariety of $X$,
then the canonical homomorphism
$$
[K(X^\prime)\otimes_{K(X)}K(X_{/Y})]_0\longrightarrow K(X^\prime_{/f^{-1}(Y)})
$$
is an isomorphism.
\end{theorem}

\begin{definition}\label{DefGi} {\em
Let $X$ be a scheme, and let $Y$ be a closed subscheme of $X$.
Following \cite[(2.9)]{HiMa68} (see also \cite[V]{HaASAV70} \cite[9.12]{Bad}), we say that:
$Y$ is $G2$ in $X$ if $K(X_{/Y})$  is a finite module over $K(X)$;
$Y$ is $G3$ in $X$ if the canonical map $K(X)\longrightarrow K(X_{/Y})$ is an isomorphism.}
\end{definition}

\begin{remark}\label{FirstRemarks}
We recall some elementary facts from \cite[(2.10)]{HiMa68}
(see also \cite[V]{HaASAV70} and \cite[Chapter 9]{Bad}).
\begin{itemize}
\item[a)] If $X$ is connected and complete over an algebraically closed field,
then $K(X_{/Y})$ is a finite direct product of fields. Hence, if $Y$ is $G3$ in $X$, then $Y$ is necessarily connected. Moreover, $G3$ $\Rightarrow$ G2.
\medskip

\item[b)] Theorem \ref{HironakaMatsumuraFormula} easily implies that, if $Y$ a closed subvariety of an irreducible
variety $X$, then  $(X^\prime,f^{-1}(Y))$ is $G3$
if and only if $(X,Y)$ is $G3$  for every proper surjective morphism
 $f\colon X^\prime\longrightarrow X$  from an irreducible variety $X^\prime$.
 See \cite[(2.7)]{HiMa68}, cf. also  \cite[V]{HaASAV70} and \cite[9.9, 9.13(i)]{Bad}.\medskip

\item[c)] The following elementary fact will be useful.
Let $X$ be an irreducible variety with two ``nested'' closed subsets $Y_2\subset Y_1\subset X$. Assuming $K(X_{/Y_1})$ a field, then $(X,Y_1)$ is $G3$ if $(X,Y_2)$ is $G3$. Indeed:
\[
\xymatrix@+1pc{
K(X) \ar[rr]^\cong \ar@{^{(}->}[dr]& & K(X_{/Y_2})\\
& K(X_{/Y_1}) \ar@{^{(}->}[ur]
}
\]
\end{itemize}
\end{remark}

The following fundamental Theorem, also due to Hironaka and Matsumura (see \cite[(3.3)]{HiMa68}), completely explains what happens in case $Y$ is a closed subscheme of a projective space
$X={\mathbb P}^n$ (as already remarked, in order $Y$ to be $G3$ in $X$, we must have $Y$ connected).

\begin{theorem}\label{HironakaMatsumuraTheorem}
Let $Y$ be a connected closed subscheme of $X={\mathbb P}^n$.
Then, $Y$ is $G3$ in $X$ if (and only if) $\dim (Y)>0$.
\end{theorem}

\begin{remark}  Let $X$ be a projective irreducible variety defined over the field of complex numbers $k=\mathbb C$,  and let $Y$ be a connected positive dimensional closed subvariety of $X$.
Using results of Chow and Serre's GAGA one can show that for every connected open subset $U$ of
$X$  (in the complex topology of $X$) containing $Y$, one has the following inclusions  (see e.g. \cite{Bad}, Chapter 10)
$$K(X)\subseteq\mathscr M(U)\subseteq K(X_{/Y}).$$
In particular, if $Y$ is $G3$ in $X$ then one gets that $K(X)=\mathscr M(U)$. Therefore Theorem \ref{HironakaMatsumuraTheorem} implies the following analytic result  of Severi-Barth: For every closed connected subvariety $Y$ of $\mathbb P^n_{\mathbb C}$ of dimension $\geq 1$, every meromorphic function $\xi$ defined in a complex connected neighborhood $U$ of $Y$ in $\mathbb P^n_{\mathbb C}$ can be (uniquely) extended to a meromorphic (and hence, rational) function on $\mathbb P^n_{\mathbb C}$. Severi proved this result in \cite{Se32}
in the case when $Y$ is a nonsingular hypersurface of $\mathbb P^n_{\mathbb C}$, and Barth generalised it to every closed connected subset $Y$ of $\mathbb P^n_{\mathbb C}$ of positive dimension in \cite{Ba68}.
\end{remark}

As we shall see in Proposition \ref{FaltingsUnConnPair}, the pairs $(X,Y)$ as in the statement of Theorem \ref{FaltingsTheorem} share a fundamental property.
They are {\it universally connected}, that is:

\begin{definition} {\em
Let $X$ be a variety over an algebraically closed field $k$, and let $Y$ be a closed subvariety of $X$.
We say that the pair $(X,Y)$ is {\em universally connected} if $f^{-1}(Y)$ is connected in $X^\prime$
for every proper surjective morphism $f\colon X^\prime\longrightarrow X$ from an irreducible variety
$X^\prime$. }
\end{definition}

According to a result of B\u adescu and Schneider (see \cite[(2.7)]{BaSch02}, cf. also \cite[9.22]{Bad}), universally connected pairs can be characterized by means of formal rational functions as follows:

\begin{theorem}
Let $X$ be an irreducible variety, and let $Y$ be a closed subvariety of $X$. the following conditions are equivalent:
\begin{itemize}
\item[(i)] $(X,Y)$ is universally connected,
\item[(ii)] $K(X_{/Y})$ is a field and $K(X)$ is algebraically closed in $K(X_{/Y})$,
\item[(iii)] $K(X_{/Y})$ is a field and the algebraic closure of $K(X)$ in $K(X_{/Y})$
is purely inseparable over $K(X)$.
\end{itemize}
\end{theorem}

\begin{remark}\label{G3UnConn}
By Example \ref{FirstRemarks}, b)
we see that if $(X,Y)$ is $G3$ (with $X$ irreducible) then $(X^{\prime},f^{-1}(Y))$ is still $G3$, hence by Remark
\ref{FirstRemarks}, a), $(X,Y)$ is universally connected.
Notice also that, If $(X,Y)$ is a universally connected pair,
then $Y$ is $G3$ in $X$ if and only if $Y$ is G2 in$X$.
\end{remark}

\begin{examples}\label{FirstExamples}
Let us recall some basic known examples and results.
\begin{enumerate}
\item[a)] A point $Y=\{P\}$ in $\mathbb{P}^n$ is never G2 (and in particular, nor $G3$) by Example \ref{ExamplePn}, a).\medskip

\item[b)] Let $f\colon X^\prime\longrightarrow X$ be a proper surjective morphism of
irreducible algebraic varieties, and let $Y\subseteq X$ and $Y^\prime\subseteq X^\prime$ be closed subvarieties such that $f(Y^\prime)\subseteq Y$.  Assume that $K(X_{/Y}), K(X^\prime_{/Y^\prime})$,
and $K(X^\prime_{/f^{-1}(Y)})$ are fields.  If $Y^\prime$ is $G3$ in $X^\prime$,
then $Y$ is $G3$ in $X$ (\cite[9.23]{Bad}).\medskip

\item[c)] If $Y^\prime$ is G2 in $X^\prime$, and if $f\colon X^\prime\longrightarrow X$
is a non constant dominant morphism of irreducible algebraic varieties, then
$\dim f(Y^\prime)>0$.\medskip

\item[d)] B\u adescu (see \cite[(0.1)]{BaNag96}, cf. also \cite[11.1]{Bad}) proved the following strengthening (and generalization) of Fulton-Hansen
connectedness Theorem (\cite{FuHa79}):  for any proper morphism
 $f\colon X^\prime\longrightarrow \mathbb{P}^n(e)\times \mathbb{P}^n(e)$
from an irreducible variety $X^\prime$,  such that $\dim f(X)>n$, then $f^{-1}(\Delta_{\mathbb{P}^n(e)})$
is $G3$ in $X^\prime$ . Here $\mathbb{P}^n(e)$ denotes the $n$-dimensional weighted projective space of weights $e=(e_0,\ldots,e_n)$, with $e_i\geq 0$, $i=0,\ldots,n$. Theorem \ref{FaltingsTheorem} above plays an important role in the proof of this result.
\end{enumerate}
\end{examples}

Both our proofs follow the pattern suggested by the two facts below. The first shows, as promised, that the pairs
$(X,Y)$ as in Theorem \ref{FaltingsTheorem} are universally connected.

\begin{proposition}\label{FaltingsUnConnPair}
Let $Y$ be a closed subvariety of a projective irreducible subvariety $X$. Assume
that $X\subset \mathbb P^n$, $\dim X=d>2$ and $Y$ is a set-theoretic intersection of
$X$ with $r$ hyperplanes of $\mathbb P^n$, with $r\le d-1$.
Then $(X,Y)$ is universally connected.
\end{proposition}

\begin{proof} Let $f\colon X'\to X$ be a proper surjective morphism. We have to prove that $f^{-1}(Y)$
is connected. By the Stein factorization we may assume that $f$ is finite. In this case $f^*(\mathscr O_X(1))$ is ample (and generated by its global sections) in which case the proposition follows from a result of Grothendieck (see \cite[\'Expos\'e XIII, Corollaire  2.2]{GroSGA2} (cf. also  \cite[Ch. 7, Corollary
7.7]{Bad}).
\end{proof}

A way to prove that a universally connected pair $(X,Y)$ is $G3$ is suggested by the following simple observation (cf. \cite[Proposition 9.23]{Bad}):

\begin{lemma}\label{Lemma}
Let $X$ be an irreducible projective variety, and let $Y$ be a closed subvariety of $X$.
Assume $(X,Y)$ is universally connected. Then $(X,Y)$ is $G3$ if and only if there is a surjective
proper morphism $f\colon X^\prime\to X$ and a closed subvariety $Y^{\prime\prime}\subset X^\prime$ such that
$f(Y^{\prime\prime})\subset Y$ and $(X^\prime,Y^{\prime\prime})$ is $G3$.
\end{lemma}
\begin{proof}
The necessity is trivial. For the converse:
being $(X,Y)$ universally connected $K(X^\prime_{/f^{-1}(Y)})$ is a field, hence,
using part b) of Remarks \ref{FirstRemarks} we get that $(X^\prime,f^{-1}(Y))$ is $G3$.
Then, by Example \ref{FirstExamples}, b), $(X,Y)$ is $G3$.
\end{proof}

\section{First proof of Theorem \ref{FaltingsTheorem}}\label{FirstProof}

We show that an idea of Hironaka and Matsumura to prove the theorem in the case when
$\dim(Y)=\dim(X)-r$ (i.e. if $Y$ is a complete intersection in $X$, see \cite[(4.3)]{HiMa68}) can be suitably modified to yield a proof of Theorem \ref{FaltingsTheorem} in general.
Our surjective proper morphism, as in Lemma \ref{Lemma},
is going to be a projection to $X$ from the closure
of the graph of a suitable linear projection of $X$.

Let $H_1$, \dots, $H_r$ in $\mathbb P^n$ be the hyperplanes cutting $Y$ on $X$
in ${\mathbb P}^n$, i.e. $Y=X\cap H_1\cap\dots\cap H_r$ (set-theoretically);
we can always assume that the $H_i$ are all distinct.

Let us consider the $(n-r)$-plane $H=H_1\cap\dots\cap H_r\supseteq Y$ and let
$L\subset H$ be a $(n-d-1)$-plane choosen in such a way  that $L$ does not contain
any irreducible component of $Y$.
Choose a $d$-plane $M\cong\mathbb P^d$, disjoint from $L$, and
let us consider the linear projection:  \[\pi\colon \mathbb P^n\dashrightarrow M,\]
of center $L$.
Setting $U:=\mathbb P^n\setminus L$ we find a morphism $\pi_U\colon U\rightarrow M$.
Note that $X_U:=X\cap U\ne \emptyset$ is an open (and so dense) subset of $X$,
and let us consider the morphism:
\[g_U:=(\pi_U)|_{X_U}\colon X_U\rightarrow M,\]
together with its graph:
\[\Gamma_U=\left\{(x,y)\in X_U\times M\mid y=g_U(x)\right\}\subseteq X_U\times
M\subseteq X\times M.\]
Let $\Gamma$ be the closure of $\Gamma_U$ in $X\times M$. So we get the commutative diagrams:
\begin{multicols}{2}\[
\xymatrix@+1pc{\Gamma_U \ar[r]^{(p_1)_U}_\cong \ar@{->>}[d]_{(p_2)_U}  & X_U \ar[dl]^{g_U}\\
M &
}
\]

\[
X^\prime:=\!\!\xymatrix@+1pc{
\Gamma \ar@{->>}[r]^{f:=p_1} \ar@{->>}[d]_{p_2} & X \ar@{-->}[dl]^{g}\\
M &
}
\]
\end{multicols}
\noindent in which  $(p_1)_U$ is an isomorphism and $(p_2)_U$ is dominant by the choice of $L$.
So, since $p_1$ and $p_2$ are projective morphism, we see that $p_1$ and $p_2$ are surjective. Moreover, being $\Gamma_U\cong X_U$, we have that $\Gamma$ is irreducible and $p_1$ is a birational regular map, this is our morphism $f$ as in Lemma \ref{Lemma}, and $X^\prime=\Gamma$.

Now we need to find a closed subvariety $Y^{\prime\prime}\subset X'=\Gamma$ which is $G3$ in $X'$ such that $f(Y^{\prime\prime})\subset Y$.
To this end, let us consider $Y^{\prime}:=\pi_U(H\cap U)=H\cap M$ (because $H\supset L$).
It is a $d-r$-plane in $M\cong {\mathbb P}^d$, and since $d-r\ge 1$, we infer that $Y^{\prime}$ is $G3$ in $M$ by
Theorem \ref{HironakaMatsumuraTheorem}. Therefore $Y^{\prime\prime}:={p_2}^{-1}(Y^{\prime})$ is $G3$ in $\Gamma$ by Theorem \ref{HironakaMatsumuraFormula}.
Clearly, by construction: $f(Y^{\prime\prime})=p_1({p_2}^{-1}(H\cap M))\subseteq H\cap X= Y$.

Now, since $Y^{\prime\prime}$ is $G3$ in $X'$, by Lemma \ref{Lemma} and Proposition \ref{FaltingsUnConnPair}, $Y$ is $G3$ in $X$. This finishes the first proof of Theorem \ref{FaltingsTheorem}.

\section{Second proof of Theorem \ref{FaltingsTheorem}}\label{SecondProof}

This second proof makes use of a suitable incidence variety.
Specifically, under the notation of Theorem \ref{FaltingsTheorem}, let $h_1, \dots, h_r$ in $\mathbb P^n$ be linear forms defining the distinct hyperplanes
$H_1, \dots, H_r$ such that, set-theoretically $Y=X\cap H_1\cap\dots\cap H_r$.
Recall that by hypothesis we have $r+1\leq d=\dim(X)$.
Let us consider the projective space:
\[P:=\mathbb P(H^0(\mathscr O_{\mathbb P^n}^{\oplus(r+1)}(1)))\cong
\mathbb P^{(n+1)(r+1)-1}.\]
Taken a global section $\sigma\in H^0(\mathscr O_{\mathbb P^n}(1))$, let us denote
by $V(\sigma)$ the zero locus of $\sigma$ in $\mathbb P^n$, which is a hyperplane in
$\mathbb P^n$ in the case that $\sigma$ is nonzero, and coincides with
$\mathbb P^n$ otherwise.
Let $Z\subset X\times P$ be the closed incidence subvariety given by:
\[Z\big\{(x,[\sigma_0,\dots,\sigma_{r}])\in X\times P\mid
x\in X\cap V(\sigma_0)\cap\dots\cap V(\sigma_{r})\big\}.\]
Let us consider the two projections:
\[
\xymatrix@+1pc{
Z \ar@{->>}[r]^{f} \ar@{->>}[d]_{g} & X \\
P &
}
\]
and note that $f$ is surjective and all its fibres are linear subspaces of $P$ of the same dimension.
Since $X$ is irreducible, we deduce that $Z$ is also irreducible.
The proper morphism $g$ is also surjective because $r+1\leq d$, and hence
$X\cap V(\sigma_0)\cap\dots\cap V(\sigma_{r})\neq \emptyset$ for each point
$[\sigma_0,\dots, \sigma_{r}]\in H^0(\mathscr O_{\mathbb P^n}^{\oplus(r+1)}(1))\setminus\{0\}$.
Moreover,
$$
g^{-1}([\sigma_0,\dots, \sigma_{r}])=Y\times \{[\sigma_0,\dots, \sigma_{r}]\}\subset f^{-1}(Y)
$$
for each point $[\sigma_0,\dots, \sigma_{r}]\in H^0(\mathscr O_{\mathbb P^n}^{\oplus(r+1)}(1))$
such that $\sigma_0,\dots, \sigma_{r} $ generate the same vector space as $h_1,\dots, h_{r} $.\medskip

As in the first proof, the theorem will be proved once we show that $g^{-1}(L)\subseteq f^{-1}(Y)$ for some connected closed subset $L\subset P$, with $\dim(L)\ge 1$. We show that we can take for $L$
a suitable line.
Let us consider the points of $P$ defined by:
$$
q_1:=[h_1,\dots,h_r,0]\textrm{ and }
q_2:=[0,h_1,\dots,h_r],
$$
and let us denote by $L\subset P$ the line joining the two points.
Since any point of $L$ is of type
$\lambda q_1+\mu q_2=[\lambda s_1,\lambda s_2+\mu s_1,\dots,\lambda s_r+\mu s_{r-1},\mu s_r]$,
we easily see that:
\[g^{-1}(L)=Y\times L\subseteq f^{-1}(Y)\subset Z.\]
This concludes the second proof.

\bibliographystyle{amsalpha}
\bibliography{FaltingsBib}

\providecommand{\bysame}{\leavevmode\hbox to3em{\hrulefill}\thinspace}
\providecommand{\MR}{\relax\ifhmode\unskip\space\fi MR }
\providecommand{\MRhref}[2]{%
  \href{http://www.ams.org/mathscinet-getitem?mr=#1}{#2}
}
\providecommand{\href}[2]{#2}
\begin{thebibliography}{B{\u{a}}d04}

\bibitem[B{\u{a}}d96]{BaNag96}
Lucian B{\u{a}}descu, \emph{Algebraic {B}arth-{L}efschetz theorems}, Nagoya
  Math. J. \textbf{142} (1996), 17--38.

\bibitem[B{\u{a}}d04]{Bad}
\bysame, \emph{Projective geometry and formal geometry}, Instytut Matematyczny
  Polskiej Akademii Nauk. Monografie Matematyczne (New Series) [Mathematics
  Institute of the Polish Academy of Sciences. Mathematical Monographs (New
  Series)], vol.~65, Birkh\"auser Verlag, Basel, 2004.

\bibitem[Bar68]{Ba68}
Wolf Barth, \emph{Fortsetzung, meromorpher {F}unktionen in {T}ori und
  komplex-projektiven {R}\"aumen}, Invent. Math. \textbf{5} (1968), 42--62.

\bibitem[BS02]{BaSch02}
Lucian B{\u{a}}descu and Michael Schneider, \emph{Formal functions,
  connectivity and homogeneous spaces}, Algebraic geometry, de Gruyter, Berlin,
  2002, pp.~1--23.

\bibitem[Fal80]{Fa80}
Gerd Faltings, \emph{A contribution to the theory of formal meromorphic
  functions}, Nagoya Math. J. \textbf{77} (1980), 99--106.

\bibitem[FH79]{FuHa79}
William Fulton and Johan Hansen, \emph{A connectedness theorem for projective
  varieties, with applications to intersections and singularities of mappings},
  Ann. of Math. (2) \textbf{110} (1979), no.~1, 159--166. \MR{MR541334
  (82i:14010)}

\bibitem[Gro60]{EGAI}
A.~Grothendieck, \emph{\'{E}l\'ements de g\'eom\'etrie alg\'ebrique. {I}. {L}e
  langage des sch\'emas}, Inst. Hautes \'Etudes Sci. Publ. Math. (1960), no.~4,
  228.

\bibitem[Gro61]{EGA}
\bysame, \emph{\'{E}l\'ements de g\'eom\'etrie alg\'ebrique. {III}. \'{E}tude
  cohomologique des faisceaux coh\'erents. {I}}, Inst. Hautes \'Etudes Sci.
  Publ. Math. (1961), no.~11, 167.

\bibitem[Gro05]{GroSGA2}
Alexander Grothendieck, \emph{Cohomologie locale des faisceaux coh\'erents et
  th\'eor\`emes de {L}efschetz locaux et globaux ({SGA} 2)}, Documents
  Math\'ematiques (Paris) [Mathematical Documents (Paris)], 4, Soci\'et\'e
  Math\'ematique de France, Paris, 2005, S\'eminaire de G\'eom\'etrie
  Alg\'ebrique du Bois Marie, 1962, Augment\'e d'un expos\'e de Mich\`ele
  Raynaud. [With an expos\'e by Mich\`ele Raynaud], With a preface and edited
  by Yves Laszlo, Revised reprint of the 1968 French original.

\bibitem[Har70]{HaASAV70}
Robin Hartshorne, \emph{Ample subvarieties of algebraic varieties}, Notes
  written in collaboration with C. Musili. Lecture Notes in Mathematics, Vol.
  156, Springer-Verlag, Berlin, 1970.

\bibitem[Har77]{Ha}
\bysame, \emph{Algebraic geometry}, Springer-Verlag, New York, 1977, Graduate
  Texts in Mathematics, No. 52.

\bibitem[HM68]{HiMa68}
Heisuke Hironaka and Hideyuki Matsumura, \emph{Formal functions and formal
  embeddings}, J. Math. Soc. Japan \textbf{20} (1968), 52--82.

\bibitem[Sev32]{Se32}
F.~Severi, \emph{{Alcune propriet\`a fondamentali dell'insieme dei punti
  singolari di una funzione analitica di pi\`u variabili.}}, {Memorie Reale
  Acad. d'Italia, Roma, 3, Mat. Nr. 1, 20 p }, 1932 (Italian).

\end{thebibliography}

%
%

\end{document}